\documentclass[3p,times]{elsarticle}

\usepackage{lineno,hyperref}
\modulolinenumbers[5]

\usepackage{amsmath,amssymb,amsfonts}
\usepackage[english]{babel}
\usepackage{graphicx}
\usepackage{amsthm}
\usepackage{mathtools}
\usepackage{ctable}
\usepackage{latexsym}
\usepackage{subfigure}

\newtheorem{theorem}{Theorem}[section]
\newtheorem{corollary}{Corollary}[section]
\newtheorem{lemma}{Lemma}[section]
\newtheorem{proposition}{Proposition}[section]

\theoremstyle{definition}
\newtheorem{definition}{Definition}
\newtheorem{remark}{Remark}
\newtheorem{example}{Example}

\journal{Journal of \LaTeX\ Templates}

\begin{document}

\begin{frontmatter}

\title{On Tsallis extropy with an application to pattern recognition}


\author[mymainaddress]{Narayanaswamy Balakrishnan}

\author[mysecondaryaddress]{Francesco Buono\corref{mycorrespondingauthor}}
\cortext[mycorrespondingauthor]{Corresponding author}
\ead{francesco.buono3@unina.it}
\author[mysecondaryaddress]{Maria Longobardi}

\address[mymainaddress]{McMaster University, Canada}
\address[mysecondaryaddress]{Università di Napoli Federico II, Italy}

\begin{abstract}
Recently, a new measure of information called extropy has been introduced by Lad, Sanfilippo and Agrò as the dual version of Shannon entropy. In the literature, Tsallis introduced a measure for a discrete random variable, named Tsallis entropy, as a generalization of Boltzmann-Gibbs statistics. In this work, a new measure of discrimination, called Tsallis extropy, is introduced and some of its properties are then discussed. The relation between Tsallis extropy and entropy is given and some bounds are also presented. Finally, an application of this extropy to pattern recognition is demonstrated.
\end{abstract}


\end{frontmatter}

\linenumbers

\section{Introduction}
Let $X$ be a discrete random variable with support $S=\{x_1,\dots,x_N\}$ and with corresponding probability vector $\bold p=(p_1,\dots,p_N)$, i.e., $\mathbb P(X=x_i)=p_i$, for $i=1,\dots,N$. In 1948, Shannon \cite{shannon} introduced a measure of information related to the information content and the uncertainty about an event associated with a discrete random variable. This measure, come to be known as Shannon entropy, is defined as
\begin{equation}
H(X)=-\sum_{i=1}^N p_i \log p_i,
\end{equation}
where $\log$ is the natural logarithm. The concept of entropy has since been generalized in different ways. Analogous to the discrete case, the Shannon entropy has been defined in the continuous case as well as
\begin{equation}
\nonumber
H(X)=-\int_0^{+\infty} f(x)\log f(x)\mathrm dx,
\end{equation}
where $X$ is a non-negative random variable with probability density function $f$. Although the definitions are similar, the entropy is always non-negative in the discrete case while it could be negative in the continuous case. In the literature, different versions of entropy have been introduced, including residual and past entropy \cite{dicrescenzo, ebrahimi}, cumulative residual and past entropy \cite{dicrescenzo3, dicrescenzo4, rao,tahmasebi2}, weighted entropies \cite{dicrescenzo2}, dynamic entropies \cite{asadi}, weighted cumulative residual entropy \cite{mirali}, generalized dynamic entropies \cite{sekeh}, and also some relationships with other concepts in reliability theory \cite{balakrishnan2,longobardi,tahmasebi}.

Lad \emph{et al.} \cite{lad} introduced the extropy, a measure of uncertainty, as a dual version of the entropy. It is useful for comparing the uncertainty between two random variables $X$ and $Y$, i.e., if the extropy of $X$ is less than that of $Y$, then $X$ has less uncertainty than $Y$. For a discrete random variable $X$, the extropy $J(X)$ is defined as
\begin{equation}
\label{eq4}
J(X)=-\sum_{i=1}^N (1-p_i)\log(1-p_i),
\end{equation}
and it is always non-negative. The extropy has subsequently been widely studied and several different versions have also been proposed in the literature; see \cite{balakrishnan, jahan, kamari, qiu}.

Among the different generalizations of Shannon entropy, the Tsallis entropy \cite{tsallis} has attracted considerable attention. For a discrete random variable, the Tsallis entropy $S_{\alpha}(X)$ is defined as
\begin{equation}
\label{eq3}
S_{\alpha}(X)=\frac{1}{\alpha-1}\left(1-\sum_{i=1}^N p_i^{\alpha}\right),
\end{equation}
where $\alpha>0$ and $\alpha\ne1$. It is a generalization of Shannon entropy since it is evident that
\begin{equation}
\label{eq5}
\lim_{\alpha\to1}S_{\alpha}(X)=H(X).
\end{equation}

In the present work, we introduce a measure of uncertainty dual to the Tsallis entropy in \eqref{eq3}, and it may be referred to as the Tsallis extropy. The rest of the paper proceeds as follows. In Section \ref{sec2}, the Tsallis extropy is defined, and some of its properties are given, and its relationships to other known measures are described. In Section \ref{sec3}, we study the maximum Tsallis extropy and establish an upper bound for it. In Section \ref{sec4}, we apply this new measure to a problem in pattern recognition, and compare some other known methods with the method based on the extropy. Finally, in Section \ref{sec5}, we provide some concluding remarks and summarize the results of this work.

\section{Tsallis extropy}
\label{sec2}

In this section, we introduce the Tsallis extropy, dual to the Tsallis entropy, as a new measure of uncertainty. It is defined to preserve a relationship similar to the one between Shannon entropy and extropy. It is important to mention at this point that Lad \emph{et al.} \cite{lad} proved the following property with regard to the sum of entropy and extropy:
\begin{equation}
\label{eq1}
H(\bold p)+J(\bold p)=\sum_{i=1}^N H(p_i,1-p_i)=\sum_{i=1}^N J(p_i,1-p_i),
\end{equation}
where $H(p_i,1-p_i)=J(p_i,1-p_i)=-p_i\log p_i - (1-p_i)\log(1-p_i)$. Observe that the two terms on the RHS of this expression are the entropy and extropy of a discrete random variable taking on two values with the corresponding probabilities as $(p_i,1-p_i)$.

\begin{definition}
Let $X$ be a discrete random variable with support $S=\{x_1,\dots,x_N\}$ and with corresponding probability vector $\bold p=(p_1,\dots,p_N)$, and let $\alpha>0$, $\alpha\ne1$. Then, the Tsallis extropy of $X$, $JS_{\alpha}(X)$, is defined as
\begin{equation}
\label{eq2}
JS_{\alpha}(X)=\frac{1}{\alpha-1}\left(N-1-\sum_{i=1}^N (1-p_i)^{\alpha}\right).
\end{equation}
\end{definition}

\begin{remark}
The definition in \eqref{eq2} is obtained in a different way from the Tsallis entropy in \eqref{eq3}. In fact, using the normalization condition, we can rewrite the Tsallis entropy as
\begin{equation}
\nonumber
S_{\alpha}(X)=\frac{1}{\alpha-1}\sum_{i=1}^N p_i(1-p_i^{\alpha-1}).
\end{equation}
We can then introduce the Tsallis extropy as
\begin{equation}
\label{eq9}
\frac{1}{\alpha-1}\sum_{i=1}^N (1-p_i)\left(1-(1-p_i)^{\alpha-1}\right)=\frac{1}{\alpha-1}\left(\sum_{i=1}^N (1-p_i)-\sum_{i=1}^N (1-p_i)^{\alpha}\right),
\end{equation}
from which the definition in \eqref{eq2} follows readily. Intuitively, the extropy corresponding to a fixed entropy could be simply introduced by replacing all the $p_i$ by $(1-p_i)$ as seen in \eqref{eq9}, but we will see in Proposition \ref{prsum} that our definition has a deeper meaning as it preserves the invariance property about the sum of entropy and extropy \eqref{eq1} studied by Lad \emph{et al.} \cite{lad}.
\end{remark}

\begin{proposition}
\label{prop1}
The Tsallis extropy in non-negative.
\end{proposition}

\begin{proof}
Let us consider the expression of Tsallis extropy in the LHS of \eqref{eq9}. For $\alpha>1$, the function $h(x)=x^{\alpha-1}$ is increasing in $x>0$, and so
\begin{equation}
\nonumber
1-(1-p_i)^{\alpha-1}\ge 0,
\end{equation}
and hence the Tsallis extropy is non-negative. For $0<\alpha<1$, the function $h(x)=x^{\alpha-1}$ is decreasing in $x>0$, i.e.,
\begin{equation}
\nonumber
1-(1-p_i)^{\alpha-1}\leq 0,
\end{equation}
and hence the Tsallis extropy is non-negative due to the multiplicative factor $\frac{1}{\alpha-1}$ being negative.
\end{proof}

\begin{remark}
Proposition \ref{prop1} is logical since the situation characterized by the lowest uncertainty is the one in which we have $\bold p=(1,0,\dots,0)$, which corresponds to $JS_{\alpha}(X)=0$.
\end{remark}

We now present some examples to demonstrate the evaluation of Tsallis extropy.

\begin{example}
\label{ex1}
Let $X$ be a discrete random variable uniformly distributed over $\{1,\dots,N\}$. Then, the Tsallis extropy is given by
\begin{eqnarray}
\nonumber
JS_{\alpha}(X)&=&\frac{1}{\alpha-1}\left\{N-1-\sum_{i=1}^N\left(1-\frac{1}{N}\right)^{\alpha}\right\} \\
\nonumber
&=& \frac{1}{\alpha-1}\left\{N-1-\frac{(N-1)^{\alpha}}{N^{\alpha-1}}\right\} \\
\label{eq7}
&=& \frac{N-1}{\alpha-1} \frac{\{N^{\alpha-1}-(N-1)^{\alpha-1}\}}{N^{\alpha-1}}.
\end{eqnarray}
\end{example}

In the following proposition, we show that the Tsallis extropy reduces to the extropy in \eqref{eq4} when $\alpha$ tends to $1$. Bear in mind that this is a classical property of Tsallis and Shannon entropies.

\begin{proposition}
Let $X$ be a discrete random variable with finite support $S$ and with corresponding probability vector $\bold p$. Then,
\begin{equation}
\lim_{\alpha\to 1} JS_{\alpha}(X)=J(X).
\end{equation}
\end{proposition}

\begin{proof}
From \eqref{eq2} and by using L'Hôpital's rule, we find
\begin{eqnarray}
\nonumber
\lim_{\alpha\to 1} JS_{\alpha}(X)&=&\lim_{\alpha\to1} \frac{1}{\alpha-1}\left(N-1-\sum_{i=1}^N (1-p_i)^{\alpha}\right) \\
\nonumber
&=&-\lim_{\alpha\to1} \sum_{i=1}^N (1-p_i)^{\alpha}\log(1-p_i) \\
\nonumber
&=& -\sum_{i=1}^N (1-p_i)\log(1-p_i)=J(X).
\end{eqnarray}
\end{proof}

Next, for discussing the sum of Tsallis entropy and extropy, similar to the one presented in \eqref{eq1}, we need the following lemma about random variables taking on two values.

\begin{lemma}
\label{lem1}
Let $X$ be a discrete random variable taking on two values with corresponding probabilities $(p,1-p)$. Then,
\begin{equation}
JS_{\alpha}(X)=S_{\alpha}(X).
\end{equation}
\end{lemma}

\begin{proof}
From \eqref{eq2}, with $N=2$, we find
\begin{eqnarray}
\nonumber
JS_{\alpha}(X)&=&\frac{1}{\alpha-1}\left[1-(1-p)^{\alpha}-\{1-(1-p)\}^{\alpha}\right] \\
\nonumber
&=&\frac{1}{\alpha-1}\left\{1-p^{\alpha}-(1-p)^{\alpha}\right\}=S_{\alpha}(X).
\end{eqnarray}
\end{proof}

\begin{proposition}
\label{prsum}
Let $X$ be a discrete random variable with finite support $S$ and with corresponding probability vector $\bold p$. Then,
\begin{equation}
S_{\alpha}(X)+JS_{\alpha}(X)=\sum_{i=1}^N S_{\alpha}(p_i,1-p_i)=\sum_{i=1}^N JS_{\alpha}(p_i,1-p_i),
\end{equation}
where $S_{\alpha}(p_i,1-p_i)$ and $JS_{\alpha}(p_i,1-p_i)$ are the Tsallis entropy and extropy of a discrete random variable taking on two values with corresponding probabilities $(p_i,1-p_i)$.
\end{proposition}

\begin{proof}
We have to prove only the first equality, since the second one is given in Lemma \ref{lem1}. From Equations \eqref{eq3}--\eqref{eq2}, we have
\begin{eqnarray}
\nonumber
S_{\alpha}(X)+JS_{\alpha}(X)&=&\frac{1}{\alpha-1}\left\{N-\sum_{i=1}^N p_i^{\alpha} -\sum_{i=1}^N (1-p_i)^{\alpha}\right\} \\
\nonumber
&=&\frac{1}{\alpha-1}\sum_{i=1}^N \left\{1-p_i^{\alpha} - (1-p_i)^{\alpha}\right\} \\
\nonumber
&=& \sum_{i=1}^N S_{\alpha}(p_i,1-p_i),
\end{eqnarray}
as required.
\end{proof}

In the following proposition, we will show that for the choice of the parameter $\alpha=2$, the Tsallis entropy and extropy coincide.

\begin{proposition}
Let $X$ be a discrete random variable with finite support $S$ of cardinality $N$. Then, $S_2(X)=JS_2(X)$.
\end{proposition}

\begin{proof}
From \eqref{eq2}, upon choosing $\alpha=2$, we obtain
\begin{eqnarray}
\nonumber
JS_{2}(X)&=&\frac{1}{2-1}\left(N-1-\sum_{i=1}^N (1-p_i)^{2}\right) \\
\nonumber
&=&N-1-\sum_{i=1}^N (1+p_i^2-2p_i) \\
\nonumber
&=& N-1-N-\sum_{i=1}^N p_i^2 +2 \sum_{i=1}^N p_i \\
\nonumber
&=& 1-\sum_{i=1}^N p_i^2= S_2(X),
\end{eqnarray}
as required.
\end{proof}

In the following theorem, we will prove that the Tsallis entropy is always greater than the Tsallis extropy for $\alpha<2$ and that the reverse inequality holds for $\alpha>2$.
\begin{theorem}
For any discrete random variable $X$ with support of cardinality $N\geq3$, we have
\begin{eqnarray}
\nonumber
S_{\alpha}(X)&\geq& JS_{\alpha}(X) \ \ \mbox{if} \ \ 0<\alpha<2, \\
\nonumber
S_{\alpha}(X)&\leq& JS_{\alpha}(X) \ \ \mbox{if} \ \ \alpha>2. 
\end{eqnarray}
\end{theorem}

\begin{proof}
First of all, we remark that for $\alpha=1$ we mean the limit case in which we obtain the well known result about entropy and extropy. Let us consider the difference between Tsallis entropy and extropy given by
\begin{equation}
\nonumber
S_{\alpha}(X)-JS_{\alpha}(X)=\frac{1}{\alpha-1}\left[2-N-\sum_{i=1}^N p_i^{\alpha}+\sum_{i=1}^N (1-p_i)^{\alpha}\right].
\end{equation}
Then, we consider the Lagrange function $L$ defined as
\begin{equation}
\nonumber
L=S_{\alpha}(X)-JS_{\alpha}(X)+\lambda\left(\sum_{i=1}^N p_i -1\right),
\end{equation}
for which the partial derivatives with respect to $p_i$ are
\begin{equation}
\nonumber
\frac{\partial L}{\partial p_i} =\frac{-\alpha}{\alpha-1} \left(p_i^{\alpha-1}+(1-p_i)^{\alpha-1}\right)+\lambda,
\end{equation}
which vanish if and only if
\begin{equation}
\label{eqcost}
p_i^{\alpha-1}+(1-p_i)^{\alpha-1}=C,
\end{equation}
where $C$ is a constant. Then, we consider, for $0\leq x\leq 1$, the function $h(x)=x^{\alpha-1}+(1-x)^{\alpha-1}$, such that $h(x)=h(1-x)$ and $h(0)=h(1)=1$. The function $h$ has a minimum at $x=\frac{1}{2}$, if $\alpha>2$ or $0<\alpha<1$, and a maximum at the same point for $1<\alpha<2$. Then, in order to satisfy both \eqref{eqcost} and the normalization condition, we have only two possibilities. The first one is given by choosing one $p_i$ equal to $1$ and all the others equal to $0$, whereas the second one is given by $p_i=\frac{1}{N}$, $i=1,\dots,N$. These are the cases in which the difference between Tsallis entropy and extropy takes the maximum and the minimum values. In the first case, we have $S_{\alpha}(X)-JS_{\alpha}(X)=0$. In the second case, we obtain
\begin{eqnarray}
\nonumber
S_{\alpha}(X)-JS_{\alpha}(X)&=&\frac{1}{\alpha-1}\left[2-N-\sum_{i=1}^N \frac{1}{N^{\alpha}}+\sum_{i=1}^N \left(1-\frac{1}{N}\right)^{\alpha}\right] \\
\label{diffuni}
&=& \frac{2N^{\alpha-1}-N^{\alpha}-1+(N-1)^{\alpha}}{(\alpha-1)N^{\alpha-1}}.
\end{eqnarray}
Let us consider the numerator of \eqref{diffuni} as a function of $\alpha$, $g(\alpha)=2N^{\alpha-1}-N^{\alpha}-1+(N-1)^{\alpha}$. We have
\begin{equation}
\nonumber
g'(\alpha)=N^{\alpha-1}(\log N )(2-N)+(N-1)^{\alpha}\log(N-1),
\end{equation}
which is non negative if, and only if,
\begin{equation}
\nonumber
\alpha\leq\frac{\log\left[\frac{N}{N-2}\frac{\log(N-1)}{\log N}\right]}{\log\left(\frac{N}{N-1}\right)}=G(N).
\end{equation}
We have that, for $N\geq3$,
\begin{equation}
\label{confronto}
1<G(N)<2 \Longleftrightarrow \frac{N-2}{N-1}<\frac{\log(N-1)}{\log N} < \frac{N(N-2)}{(N-1)^2},
\end{equation}
which holds as one can see in Figure \ref{figconf}.

Then, the function $g$ has a maximum between $1$ and $2$ and $g(1)=g(2)=0$. Hence, we have $g(\alpha)>0$ if $1<\alpha<2$ and $g(\alpha)<0$ if $0<\alpha<1$ or $\alpha>2$. By recalling the definition of $g$ and \eqref{diffuni}, we obtain that the difference between Tsallis entropy and extropy for uniform distribution is greater than $0$ if $0<\alpha<1$ or $1<\alpha<2$ and less than $0$ if $\alpha>2$. Hence, $S_{\alpha}(X)-JS_{\alpha}(X)$ has minimum of $0$ and maximum for the uniform distribution if $0<\alpha<2$ and viceversa if $\alpha>2$.
\end{proof}

\begin{figure}
    \centering
         \includegraphics[scale=0.45]{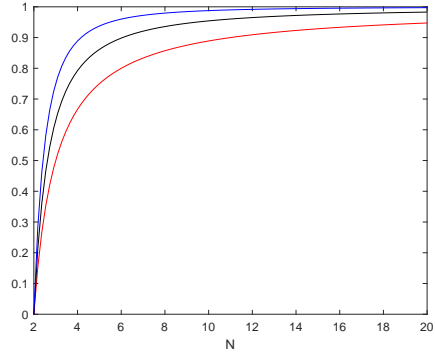} 
        \caption{Plot of the functions on the RHS of \eqref{confronto} in red, black and blue, respectively.}
        \label{figconf}
\end{figure}

\section{Maximum Tsallis extropy}
\label{sec3}

When we deal with a measure of information, it will be useful to know what its maximum value is. The Tsallis extropy reaches its maximum value when the random variable $X$ is uniformly distributed, as established in the following theorem.

\begin{theorem}
\label{thm1}
Let $X$ be a discrete random variable with finite support $S$ of cardinality $N$, and let $\alpha>0$, $\alpha\ne1$. Then, $X$ has maximum Tsallis extropy for fixed $N$ and $\alpha$ if, and only if, it is uniformly distributed.
\end{theorem}

\begin{proof}
Let $N$ and $\alpha$ be fixed. Then, we need to maximize the function of $N$ variables given by
\begin{equation}
\nonumber
JS_{\alpha}\left(\bold p\right)=\frac{1}{\alpha-1}\left(N-1-\sum_{i=1}^N (1-p_i)^{\alpha}\right),
\end{equation}
subject to the condition
\begin{equation}
\label{eq6}
\sum_{i=1}^N p_i=1.
\end{equation}
For this purpose, let us consider the Lagrange function defined by
\begin{equation}
\nonumber
JS_{\alpha}^*\left(\bold p\right)=\frac{1}{\alpha-1}\left(N-1-\sum_{i=1}^N (1-p_i)^{\alpha}\right)+\lambda\left(\sum_{i=1}^N p_i-1\right),
\end{equation}
from which we readily find the partial derivatives respect to $p_i$, $i=1,\dots,N$, as
\begin{equation}
\nonumber
\frac{\partial JS_{\alpha}^*\left(\bold p\right)}{\partial p_i}=\frac{\alpha}{\alpha-1}(1-p_i)^{\alpha-1}+\lambda.
\end{equation}
Then, we can determine the stationary points as
\begin{equation}
\nonumber
\frac{\alpha}{\alpha-1}(1-p_i)^{\alpha-1}+\lambda=0 \Longleftrightarrow p_i=1-\left(\frac{(1-\alpha)\lambda}{\alpha}\right)^{\frac{1}{\alpha-1}}=K,
\end{equation}
where $K$ is a constant. For satisfying the condition in \eqref{eq6}, we need to have $K=\frac{1}{N}$, in which case $\bold p$ becomes the probability mass function vector of a discrete uniform distribution.
\end{proof}

\begin{remark}
In Example \ref{ex1}, the Tsallis extropy has been evaluated for the uniform distribution over $N$ elements, and so the maximum Tsallis extropy is as presented in \eqref{eq7}.
\end{remark}

\begin{theorem}
\label{thm2}
The Tsallis extropy is less than $1$.
\end{theorem}

\begin{proof}
To establish this result, we show that the Tsallis extropy of a discrete uniform distribution increases to $1$ as the size of the support $N$ increases. Let $X_N$ be a discrete random variable uniformly distributed over a finite support of size $N$. From \eqref{eq7}, we know the corresponding Tsallis extropy is
\begin{equation}
\nonumber
JS_{\alpha}(X_N)= \frac{1}{\alpha-1}\left(N-1- \frac{(N-1)^{\alpha}}{N^{\alpha-1}}\right).
\end{equation}
Let us now consider the function
\begin{equation}
\label{eq8}
g(N)=N-1- \frac{(N-1)^{\alpha}}{N^{\alpha-1}},
\end{equation}
and show that it increases for $\alpha>1$ and decreases for $0<\alpha<1$. This way, we will prove that $JS_{\alpha}(X_N)$ is increasing in $N$. Let us consider the derivative of $g$, by treating it as a function of a continuous variable $N$, given by
\begin{equation}
\nonumber
g'(N)=1-\frac{(N-1)^{\alpha-1}}{N^{\alpha}}(\alpha+N-1)=\frac{N^{\alpha}-(N-1)^{\alpha-1}(\alpha+N-1)}{N^{\alpha}},
\end{equation}
whose sign is determined by
\begin{equation}
\nonumber
N^{\alpha}-(N-1)^{\alpha-1}(\alpha+N-1)=N^{\alpha}-(N-1)^{\alpha}-\alpha(N-1)^{\alpha-1},
\end{equation}
which, by mean value theorem, is equal to
\begin{equation}
\nonumber
\alpha(N-1+\varepsilon)^{\alpha-1}-\alpha(N-1)^{\alpha-1},
\end{equation}
for some $\varepsilon\in(0,1)$. Thence, by using the fact that the function $h(x)=x^{\alpha-1}$ is incresing in $x>0$ for $\alpha>1$ and decreasing for $0<\alpha<1$, we get the monotonicity of $g(N)$ in \eqref{eq8}.

Now, we evaluate the limit of $JS_{\alpha}(X_N)$ as $N$ tends to infinity. We have
\begin{eqnarray}
\nonumber
\lim_{N\to+\infty} \frac{N-1}{\alpha-1} \frac{(N^{\alpha-1}-(N-1)^{\alpha-1})}{N^{\alpha-1}}&=& \lim_{N\to+\infty} \frac{N-1}{\alpha-1}\left[1-\left(1-\frac{1}{N}\right)^{\alpha-1}\right] \\
\nonumber
&=&  \lim_{N\to+\infty} \frac{N-1}{N}=1.
\end{eqnarray}
Finally, upon using the result in Theorem \ref{thm1} about the maximum Tsallis extropy, we conclude that the Tsallis extropy is less than $1$ for any discrete random variable.
\end{proof}

\begin{corollary}
For any discrete random variable $X$, we have
\begin{equation}
\nonumber
0\leq JS_{\alpha}(X)<1.
\end{equation}
\end{corollary}

\begin{proof}
The result follows readily from Proposition \ref{prop1} and Theorem \ref{thm2}.
\end{proof}

\section{Application to pattern recognition}
\label{sec4}

In this section, we give an application of the Tsallis extropy in pattern recognition by using the well-known Iris dataset in \cite{dhreeru}. We then compare the results obtained with those in the Dempster-Shafer theory of evidence \cite{dempster} due to Kang et al. \cite{kang2} and Buono and Longobardi \cite{buono}. The objective is to classify among three classes of flowers: Iris Setosa (Se), Iris Versicolour (Ve) and Iris Virginica (Vi). The dataset consists of $150$ samples, with $50$ in each class. The characteristics measured for each flower are: the sepal length in cm (SL), the sepal width in cm (SW), the petal length in cm (PL), the petal width in cm (PW) and the class (one of Se, Ve and Vi). We select 40 samples for each kind of Iris and then we find a sample of max-min value to generate a model of interval numbers, as shown in Table \ref{tab1}. Each element of the dataset can be regarded as an unknown test sample. Suppose the selected sample data is (6.1, 3.0, 4.9, 1.8, Vi).

\begin{table}[h!]
\centering
\caption{(a) The interval numbers of the statistical model. (b) Probability distributions based on Kang's method.}
\subtable[]{
\label{tab1}
\centering
\begin{tabular}{ccccc}
\toprule
\textbf{Item}	& \textbf{SL} &	\textbf{SW}  & \textbf{PL}  & \textbf{PW}  \\
\midrule
$Se$         & [4.4,5.8]       &	[2.3,4.4]  & [1.0,1.9]  &   [0.1,0.6]     \\
$Ve$           & [4.9,7.0]        &  [2.0,3.4]	  & [3.0,5.1]  &  [1.0,1.7]    \\
$Vi$         & [4.9,7.9]   &	[2.2,3.8]  & [4.5,6.9]  &  [1.4,2.5]    \\
\bottomrule
\end{tabular}
}
\subtable[]{
\centering
\begin{tabular}{ccccc}
\toprule
\textbf{Item}	& \textbf{SL} &	\textbf{SW}  & \textbf{PL}  & \textbf{PW}  \\
\midrule
$\mathbb P(Se)$         & 0.3058     &	0.2748 &  0.1391  & 0.1563      \\
$\mathbb P(Ve)$           & 0.4148        &  0.3516	 &  0.3801   & 0.3737    \\
$\mathbb P(Vi)$         & 0.2794  &	0.3736 &  0.4808 & 0.4700    \\
\bottomrule
\label{table3}
\end{tabular}
}
\end{table}

We then generate four discrete probability distributions using the method of Kang et al. \cite{kang2} based on the similarity of interval numbers. Given two intervals $A=[a_1,a_2]$ and $B=[b_1,b_2]$, their similarity $S(A,B)$ is defined as
\begin{equation}
\nonumber
S(A,B)=\frac{1}{1+\gamma \ D(A,B)},
\end{equation}
where $\gamma>0$ is the coefficient of support, and we used $\gamma=5$, for example. Then, $D(A,B)$, the distance between intervals $A$ and $B$, is defined to be
\begin{equation}
\nonumber
D^2(A,B)=\left[\left(\frac{a_1+a_2}{2}\right)-\left(\frac{b_1+b_2}{2}\right)\right]^2+\frac{1}{3}\left[\left(\frac{a_2-a_1}{2}\right)^2+\left(\frac{b_2-b_1}{2}\right)^2\right].
\end{equation}
To generate probability distributions, the intervals given in Table \ref{tab1} are used for interval $A$ and for interval $B$ we use singletons given by the selected sample. For each one of the four characteristics measured, we get three values of similarity and then we obtain a probability distribution by normalizing them (see Table \ref{table3}). We then evaluate the Tsallis extropy of these probability distributions, as presented in Table \ref{table3bis}, wherein we have used $\alpha=0.5,\ 0.7,\ 1.5,\ 2$.

\begin{table}[h!]
\centering
\caption{Tsallis extropy (a) and the weights (b) for different choices of $\alpha$.}
\subtable[]{
\centering
\begin{tabular}{ccccc}
\toprule
\textbf{Item}	& \textbf{SL} &	\textbf{SW}  & \textbf{PL}  & \textbf{PW}  \\
\midrule
$\alpha=0.5$            & 0.8941    & 0.8965  & 0.8715  &   0.8759   \\
$\alpha=0.7$            & 0.8560   & 0.8592  & 0.8267  &   0.8324   \\
$\alpha=1.5$            & 0.7245    & 0.7291  & 0.6781  &   0.6871   \\
$\alpha=2$            & 0.6564    & 0.6613  & 0.6050  &   0.6150  \\

\bottomrule
\label{table3bis}
\end{tabular}
}
\subtable[]{
\centering
\begin{tabular}{ccccc}
\toprule
	& $\omega(SL)$ &	$\omega(SW)$  & $\omega(PL)$  & $\omega(PW)$  \\
\midrule
$\alpha=0.5$            & 0.2476    & 0.2470  & 0.2533  &   0.2522   \\
$\alpha=0.7$            & 0.2469   & 0.2461  & 0.2542  &   0.2528   \\
$\alpha=1.5$            & 0.2450    & 0.2439  & 0.2567  &   0.2544   \\
$\alpha=2$            & 0.2445    & 0.2433  & 0.2574  &   0.2548  \\

\bottomrule
\label{tabwe}
\end{tabular}
}
\end{table}

We use the Tsallis extropies in Table \ref{table3bis} to generate other probability distributions. Observe that the higher the extropy, the higher the uncertainty, and so it would be reasonable to give more weight to observations related to characteristics with lower Tsallis extropy. We refer to the obtained Tsallis extropies as $JS_{\alpha}(SL), \ JS_{\alpha}(SW), \ JS_{\alpha}(PL),\  JS_{\alpha}(PW)$. Due to the monotonicity of the exponential function, we choose as baseline weight the function $w(x)=\mathrm e^{-x}$, and we can then obtain the weights $\omega$ by normalization. For example, for the sepal length, we have
\begin{equation}
\nonumber
\omega(SL)=\frac{\mathrm e^{-JS_{\alpha}(SL)}}{\mathrm e^{-JS_{\alpha}(SL)}+\mathrm e^{-JS_{\alpha}(SW)}+\mathrm e^{-JS_{\alpha}(PL)}+\mathrm e^{- JS_{\alpha}(PW)}}.
\end{equation}
The values of the weights are listed in Table \ref{tabwe} for different choices of the parameter $\alpha$. We determine a final probability distribution in the following way: for each kind of flower, we have four probabilities, one for a specific characteristic; we multiply the probabilities given in Table \ref{table3} by the corresponding weights and then sum the values relating to the same class. For example, the probability of the class Iris Setosa is obtained as follows:
\begin{equation}
\nonumber
\mathbb P(Se)=0.3058 \cdot \omega(SL)+ 0.2748 \cdot \omega(SW)+ 0.1391 \cdot \omega(PL)+ 0.1563 \cdot \omega(PW).
\end{equation}
Thus, by choosing $\alpha=0.5$, we obtain the final probability distribution to be
\begin{equation}
\nonumber
\mathbb P(Se)=0.2182, \ \ \mathbb P(Ve)=0.3800, \ \ \mathbb P(Vi)=0.4018,
\end{equation}
and then the decision is that the selected flower belongs to the class with the higher probability, Iris Virginica, i.e., we thus made the correct decision, in this case.

In this manner, we tested all 150 samples for different values of $\alpha$, and observed that the overall recognition rate of this method based on the Tsallis extropy to be 94.66\%. The results obtained this way are compared with the recognition rates of the methods of Kang et al. \cite{kang2} and Buono and Longobardi \cite{buono}, and these are presented in Table \ref{table5}. The proposed method is seen to present a slightly better performance in comparison to the other two methods.

\begin{table}[h!]
\caption{The recognition rates of different methods.}
\centering
\begin{tabular}{ccccc}
\toprule
\textbf{Item}	& \textbf{Se} &	\textbf{Ve}  & \textbf{Vi}  & \textbf{Overall}  \\
\midrule
Kang's method         & 100\%    &	96\% & 84\%  & 93.33\%   \\
Buono and Longobardi's method         & 100\%      &  96\%	  & 86\% & 94\%  \\
Method based on Tsallis extropy     & 100\%    &	98\% & 86\%  & 94.66\%   \\

\bottomrule
\label{table5}
\end{tabular}
\end{table}

\section{Conclusions}
\label{sec5}
In this paper we have studied the Tsallis extropy. It is a measure of uncertainty dual to the Tsallis entropy and it is introduced by preserving a classical invariance property about entropy and extropy. We have discussed some properties of the proposed measure and have given several examples. In particular, we have examined the problem of the maximum Tsallis extropy which is always of great interest in information theory. Finally, we have illustrated an application in the context of pattern recognition by using the proposed measure, and we have then compared the recognition rates with two other known methods. It will be of interest to introduce some other new measures with interesting properties and also discuss the optimal choice of parameter $\alpha$ in the proposed measure.

\section*{Acknowledgements}
Narayanaswamy Balakrishnan thanks the Natural Sciences and Engineering Research Council of Canada for funding this research through an Individual Discovery Grant. Francesco Buono and Maria Longobardi are members of the research group GNAMPA of INdAM (Istituto Nazionale di Alta Matematica) and are partially supported by MIUR-PRIN 2017, project ``Stochastic Models for Complex Systems'', no. 2017 JFFHSH. Our sincere thanks also go to the anonymous reviewers and the Editor for their useful comments on an earlier version of the manuscript which led to this improved version.

\end{document}